\documentclass{amsart}
\usepackage{amsmath}
\usepackage{graphicx}
\usepackage{amssymb}
\usepackage{epstopdf}
\usepackage{tikz-cd}
\usepackage{mathrsfs}
\usepackage{xcolor}
\usepackage{hyperref}
\usepackage{tikz}
\usepackage{tikz-cd} 
\usetikzlibrary{cd}
\usepackage{framed}
\usepackage[T1]{fontenc}
\usepackage{bm}
\usepackage{comment}

\numberwithin{equation}{section}

\newtheorem{thm}{Theorem}[section]
\newtheorem*{thm*}{Main Theorem}
\newtheorem*{thmA*}{Theorem A}
\newtheorem*{thmB*}{Theorem B}
\newtheorem{lem}[thm]{Lemma}
\newtheorem{prop}[thm]{Proposition}
\newtheorem{defn}[thm]{Definition}
\newtheorem{cor}[thm]{Corollary}

\newtheorem{remark}[thm]{Remark}

\begin{document}

\title[derived Hecke algebras]{Non-vanishing mod $p$ of Derived Hecke algebra of the multiplicative group over number field}
\author[Dohyeong Kim]{DOHYEONG KIM}
\email{dohyeongkim@snu.ac.kr}
\address{Seoul National University, 1, Gwanak-ro, Gwanak-gu, Seoul, Republic of Korea, 08826}
\author[Jaesung Kwon]{JAESUNG KWON}
\email{jaesungkwon@snu.ac.kr}
\address{Seoul National University, 1, Gwanak-ro, Gwanak-gu, Seoul, Republic of Korea, 08826}
\date{\today}

\subjclass[2020]{}

\begin{abstract} 
%Venkatesh \cite{venkatesh2016derived} 
%Venkatesh defined the derived Hecke algebra of arithmetic manifolds attached to reductive algebraic groups over a number field. By utilizing it, Venkatesh explains the phenomenon that the same Hecke eigensystem occur in the cohomology groups of several different degrees, which is called spectral degeneracy, when the given arithmetic manifold is not an algebraic variety.
%For almost all national prime $p$, we prove the refined version of the non-vanishing mod $p$ of the degree $1$ part of the derived Hecke algebra of the torus over a number field. Also we show that this non-vanishing result applies to the spectral degeneracy problem for our case. The main ingredient is the Grunwald-Wang theorem, which is a local-global principle on the powers of the units of a number field.
We investigate the derived Hecke action on the cohomology of an arithmetic manifold associated to the multiplicative group over a number field. 
The degree one part of the action is proved to be non-vanishing modulo $p$ under mild assumptions.
The main ingredient is the Grunwald--Wang theorem. 
\end{abstract}
\maketitle
\numberwithin{equation}{section}
\setcounter{tocdepth}{1}
\tableofcontents

\section{Introduction}
Let $F$ be a number field and $\mathbf{G}$ a %reductive
semisimple algebraic group over $F$. Let $\mathbb{A}_F^{(\infty)}$ be the ring of finite adeles of $F$, $S_\infty$ the quotient of $\mathbf{G}(F\otimes_\mathbb{Q}\mathbb{R})$ by a maximal compact connected subgroup of itself, and $K$ an open compact subgroup of $\mathbf{G}(\mathbb{A}_F^{(\infty)})$. Then, the associated quotient
%we define {\it arithmetic manifold} $Y_K$ attached to a triple $(\mathbf{G},S_\infty,K)$ as follows: 
%\marginpar{quotient by center is missing..}
\begin{align}\label{eq:am}
%Y_K:=
\mathbf{G}(F)\backslash S_\infty\times\mathbf{G}(\mathbb{A}_F^{(\infty)})/K
\end{align}
is known to be homeomorphic to a finite union of locally symmetric spaces; see \cite{PlatonovRapinchuk} for details.
% if there is a strong approximation on the adelic points of $\mathbf{G}$ over $F$.
%\marginpar{need citation}

Spaces like \eqref{eq:am} are special in that their singular cohomology groups are modules over Hecke algebras.
We briefly recall their definition.
Let $R$ be a commutative ring with unity and $v$ a prime of $F$. Then, the {\it abstract Hecke algebra at $v$} is defined as
\begin{equation}\label{abstract:classcial:Hecke:algebra:def}
\mathcal{H}_v^0:=\mathrm{Hom}(R[\mathbf{G}(F_v)/K_v],R[\mathbf{G}(F_v)/K_v])
\end{equation}
where $F_v$ is the completion of $F$ at $v$ and $K_v$ is the $v$-component of $K$.
The multiplication on $\mathcal{H}_v^0$ is given by convolution and the $\mathrm{Hom}$-set is taken in a suitable category of $\mathbf{G}(F_v)$-modules.

% and Then $\mathcal{H}_v^0$ acts on $H^\bullet(Y_K,R)$ and this action is called the {\it Hecke action} at $v$. 
%Let us call the image of the Hecke action in the endomorphism ring of $H^\bullet(Y_K,R)$ {\it Hecke operator}. Let $\mathbb{T}_v^0$ be the algebra generated by classical Hecke operators.
%The {\it classical Hecke algebra} of $Y_K$ is defined to be the subalgebra of the endomorphism algebra of the cohomology ring $H^\bullet(Y_K,R)$ generated by the Hecke actions on $H^\bullet(Y_K,R)$ at $v$.
Venkatesh \cite{venkatesh2016derived} introduced the derived version of \eqref{abstract:classcial:Hecke:algebra:def}:
\begin{defn}[Abstract derived Hecke algebra, Venkatesh {\cite[Definition 2.2]{venkatesh2016derived}}] 
For a prime $v$ of $F$, define
$$
\mathcal{H}_v^\bullet:=\mathrm{Ext}^\bullet_{C_{\rm c}(\mathbf{G}(F_v),R)}(R[\mathbf{G}(F_v)/K_v],R[\mathbf{G}(F_v)/K_v]),
$$
where $C_{\rm c}(\mathbf{G}(F_v),R)$ is the ring of compactly supported locally constant $R$-valued functions on $\mathbf{G}(F_v)$.
%where the functor $\mathrm{Ext}$ is taken in the category of $\mathbf{G}(F_v)$-modules with the property that every elements of each module has open stabilizer in $\mathbf{G}(F_v)$.
\end{defn}
Furthermore, the graded algebra $\mathcal{H}_v^\bullet$ acts on the total cohomology space of \eqref{eq:am} graded by the cohomological degrees.
% $H^\bullet(Y_K,R)$ (see Venkatesh \cite[\textsection 2]{venkatesh2016derived}) and this action is called the {\it derived Hecke action} at $v$. 
The {\it derived Hecke algebra} $\mathbb{T}_R^\bullet$ is defined to be the algebra generated by the graded endomorphisms of the total cohomology.
%of $Y_K$ is defined by the subalgebra of $\mathbf{End}_RH^\bullet(Y_K,R)$ generated by the derived Hecke actions on $H^\bullet(Y_K,R)$ at the prime ideal $v$ for prime ideals of $v$ whose the absolute norm is invertible in $R$. 
%This graded algebra acts on the cohomology ring $H^\bullet(Y_K,R)$ of $Y_K$ in the graded fashion.

In the preceding discussion we considered a semisimple group $\mathbf G$, but the construction can be easily adapted to the case of reductive groups; see \cite{KhareRonchetti}, for example. 

A starting point for our investigation is the conjecture proposed by Prasanna and Venkatesh \cite[\textsection\,5.4 Main Conj.]{prasanna2021automorphic}.
For our purpose, we interpret the aforementioned conjecture as a prediction of the size of the derived Hecke algebra where the coefficients are taken to be $\mathbb Q$-vector spaces. 
A natural question arises: what happens if we take coefficients over $\mathbb Z$? More precisely, does the Hecke algebra remain non-trivial after reducing it modulo a prime $p$?

In this paper, we answer the question for the multiplicative group $\mathbb G_{\mathrm{m}/F}$. 
This case has an advantage of being free of the additional layer of conjectures necessary to make sense of the Prasanna--Venkatesh conjecture in full generality. 
Given that the Prasanna--Venkatesh conjecture is straightforward \cite[\textsection 1.3]{prasanna2021automorphic} for tori, we belive that its integral refinement will hint at what should hold true more generally with integral coefficients.

With the choice $\mathbf{G}:=\mathbb{G}_{\mathrm{m}/F}$, we describe the counterpart of \eqref{eq:am}, which is a classical object in the class field theory.
%which is the multiplicative torus over $F$. 
Let $F_\mathbb{R}:=F\otimes_\mathbb{Q}\mathbb{R}$ be the real Minkowski space over $F$. Let $F_{\mathbb{R},+}^\times$ be the connected component of $1\in F_\mathbb{R}^\times$. 
Put
$$
U:=\{x=(x_\tau)_\tau\in F^\times_{\mathbb{R},+}: |x_\tau|=1\text{ for any infinite places }\tau\text{ of }F\},
$$
which is the maximal compact connected subgroup of $F_{\mathbb{R},+}^\times$.
%Note that $U$ contains the roots of unity of $F$.
For an integral ideal $\mathfrak{N}$ of $F$, put
$$
K(\mathfrak{N}):=\left(\prod_{v\mid\mathfrak{N}}1+\mathfrak{N}O_v\right)\times\prod_{v\nmid\mathfrak{N}} O_v^\times
$$ 
%Let us define the torus $Y(\mathfrak{N})$ over $F$ of level $\mathfrak{N}$ by  
and define
\begin{align}\label{eq:torus}
Y(\mathfrak{N}):=F^\times\backslash(F_\mathbb{R}^\times/U)\times\mathbb{A}_F^{(\infty),\times}/K(\mathfrak{N}).
\end{align}
We will call $Y(\mathfrak N)$ an arithmetic torus.
%In this paper, we are interested in the case of $Y(\mathfrak{N})$.

%We study the derived Hecke algebra in a refined perspective for the case of $Y(\mathfrak{N})$: 
Denote by $\mathbb{T}^\bullet=\mathbb{T}^\bullet_{\mathbb{F}_p}$ the derived Hecke algebra of $Y(\mathfrak{N})$ with coefficients in $\mathbb{F}_p$.
Then, $\mathbb{T}^\bullet$ is generated by $\mathbb{T}^0$, $\mathbb{T}^1$, and $\mathbb{T}^2$ as we will prove in Proposition \ref{property:hecke:operator:2}. 
We say that the derived Hecke algebra $\mathbb{T}^j$ is {\it non-vanishing mod $p$} if $\mathbb{T}^j$ is non-trivial.

We state our main results. The statements will be vacuous when the real dimension of $Y(\mathfrak N)$ is zero, so we will assume from now on that $F$ is neither $\mathbb Q$ nor an imaginary quadratic field.
Let $p$ be a rational prime. As mentioned earlier, $\mathbb{T}^\bullet$ %is the derived Hecke algebra of $Y(\mathfrak{N})$ with coefficients in $\mathbb{F}_p$, which 
acts on the cohomology ring $H^\bullet(Y(\mathfrak{N}),\mathbb{F}_p)$ (see Definition \ref{definition:derived:hecke:algebra} for details). 
Let $O_F$ be the ring of integers of $F$. 
Put $E(\mathfrak{N}):=F^\times\cap F_{\mathbb{R},+}^\times\cap K(\mathfrak{N})$, which is  the subgroup of totally positive elements of $O_F^\times$ congruent to $1$ modulo $\mathfrak{N}$.
Let
\begin{align}\label{eq:psi}
\Psi:\mathbb{T}^1\otimes_{\mathbb{T}^0} H^0(Y(\mathfrak{N}),\mathbb{F}_p)\to H^1(Y(\mathfrak{N}),\mathbb{F}_p),\ h\otimes c\mapsto hc
\end{align}
be the map induced by the derived Hecke action.
%Then we have the following main result:

\begin{thmA*}
If $p$ does not divide the order of $O_F^\times/E(\mathfrak{N})$, then $\Psi$ in \eqref{eq:psi} is an isomorphism of $\mathbb{T}^0$-modules.
\end{thmA*}

Two remarks are in order. First, the assumption in Theorem~A fails for only finitely many primes $p$. Second, we have been unable to obtain analogous non-vanishing result for $\mathbb{T}^j$ when $j>1$.

%First, we prove a refined version of the non-vanishing mod $p$ of $\mathbb{T}^1$ of a torus over a number field.
%determined by how the image of the global units of $F$ behaves in the residue fields of $F$,
%Let us discuss this point with more detail in the following subsection. 

Next, we interpret our result in terms of the spectral degeneracy.
This refers to the phenomenon that a Hecke eigensystem occurs in {\it multiple degrees} within the total cohomology of an arithmetic manifold such as \eqref{eq:am}.
Venkatesh \cite{venkatesh2016derived} attributes the spectral degeneracy to the non-triviality of the derived Hecke action. %algebra when a derived Hecke operator of positive degree acts non-trivially.
%, for instance, the characteristic of the coefficient ring $R$ is given by zero.
We follow him to interpret non-vanishing as spectral degenercy. %One of important consequences of non-vanishing mod $p$ of the derived Hecke algebra is spectral degeneracy on Hecke eigensystems, which can be stated as follows:

\begin{thmB*}
If $p$ does not divide the order of $O_F^\times/E(\mathfrak{N})$, then any Hecke eigensystems in $H^1(Y(\mathfrak{N}),\overline{\mathbb{F}}_p)$ occur in $H^0(Y(\mathfrak{N}),\overline{\mathbb{F}}_p)$. Also, the converse holds true.
\end{thmB*}

We note that the Hecke eigensystems occuring in $H^0(Y(\mathfrak{N}),\overline{\mathbb{F}}_p)$ are described by the class field theory in terms of unramified abelian extensions of $F$, but those in $H^i(Y(\mathfrak{N}),\overline{\mathbb{F}}_p)$ for $i>0$ are not readily accessible in terms of the arithmetic of $F$.

We outline the rest of the paper. In \textsection \ref{section:torus}, we review %\marginpar{study? be specific.} 
some basic properties of $Y(\mathfrak{N})$. For example, we specify a homeomorphism from 
% \marginpar{what does it mean to ``record a proof''?} of the known fact that 
$Y(\mathfrak{N})$ to a disjoint union of $r$-dimensional tori, where $r$ is the rank of the unit group of $F$.

In \textsection \ref{section:hecke}, we recall necessary definitions and results from \cite{venkatesh2016derived}. In particular, the explicit descriptions Hecke operators are given, which will be used in later sections. We also determine the generators of the derived Hecke algebra. We also describe cohomology groups as Hecke modules.% and $\mathbb{F}_p$-vector spaces.

In \textsection \ref{section:nonvanishing}, we specialize our interest to the case of $R=\mathbb{F}_p$ and investigate the map $\Psi$. At the end of the section, Theorems A and B will be proved. Two key ingredients are the results of \textsection\ref{section:hecke} and the Grunwald--Wang theorem.
%We outline the proofs here.
%Note that by using the specific generators of $H^0(Y(\mathfrak{N}),\mathbb{F}_p)$, $\mathbb{T}^0$, and $\mathbb{T}^1$, obtained in \textsection \ref{section:hecke}, we compute the dimension of both the domain and the image of $\Psi$, and prove the injectivity of $\Psi$. We apply Grunwald-Wang theorem to prove that $\Psi$ is surjective (Theorem A).
% for almost \marginpar{specific about $p$ if possible} all $p$ .
%By applying the property of the ring $\mathbb{T}^\bullet$ and Theorem A, we obtain the spectral degeneracy (Theorem B).

To conclude the introduction, we mention a variant \cite{Harris2019} of the aforementioned conjecture \cite[\textsection\,5.4 Main Conj.]{prasanna2021automorphic}, also known as the weight one form case, where units in number fields play a central role. See \cite{Darmon2022} for its resolution in the case when the form in question is assumed to be dihedral.

\section{Torus over number field}\label{section:torus}
%\marginpar{Say this section is preparatory?}Let $\mathfrak{N}$ be a non-zero integral ideal of $F$.
Recall from \eqref{eq:torus} that $Y(\mathfrak N)$ denotes the union of topological tori associated to an integral ideal $\mathfrak N$.
%$$
%Y(\mathfrak{N})=F^\times\backslash(F_\mathbb{R}^\times/U)\times\mathbb{A}_F^{(\infty),\times}/K(\mathfrak{N}).
%$$
%At a moment, \marginpar{At a moment?} let us give some notations:
Here we set up basic notation related to $Y(\mathfrak N)$.
%Let $\mu_F$ be the group of roots of unity of $F$. We would like to take $\mathfrak{N}$ such that $E(\mathfrak{N})\cap\mu_F$ is trivial, i.e., $E(\mathfrak{N})$ is torsion-free so that 
%Since $E(\mathfrak{N})\backslash F_{\mathbb{R},+}^\times/U$ is a manifold.
%, not an orbifold:
%By Chevalley \textcolor{red}{(Citation)}, there is an ideal $\mathfrak{N}_0$ such that $E(\mathfrak{N}_0)$ contains no root of unity of $F$. From now on, we assume that $\mathfrak{N}$ is divisible by $\mathfrak{N}_0$ so that $E(\mathfrak{N})$ does not contain any roots of unity of $F$. 
%\begin{lem}\label{lemma:rootofunity}
%For the prime ideals $v$ of $F$, there exist integers $e_v>0$ such that
%if $v^{e_v}$ divides $\mathfrak{N}$ for some $v$, then the group $E(\mathfrak{N})$ is torsion-free.
%In particular, $e_v=1$ if $v$ does not divide $|\mu_F|$.
%\end{lem}
%\begin{proof}
%Let $v$ be a prime ideal of $F$ coprime to $|\mu_F|$. Then the $|\mu_F|$-power map is an automorphism on $1+vO_v$, so that $\zeta^{|\mu_F|}=1$ for $\zeta\in\mu_F\cap (1+vO_v)$ implies that $\zeta=1$.
%On the other hand, let $v$ be a prime ideal of $F$ which divides $|\mu_F|$. By Krull intersection theorem, 
%$$\mu_F\cap \bigcap_{e>0} 1+v^eO_v$$ is trivial. So we are done.
%\end{proof}
%\begin{remark}
%From {\rm Lemma \ref{lemma:rootofunity}}, we conclude that there are only finite many integral ideals $\mathfrak{N}$ of $F$ such that $E(\mathfrak{N})\cap\mu_F$ is non-trivial.
%\end{remark}
Let $r_1$ and $r_2$ be the number of real and complex places of $F$, respectively. Let $r$ be the rank of $O_F^\times$, which is equal to $r_1+r_2-1$ by the Dirichlet unit theorem.
%From now on, let us assume that $v^{e_v}$ divide $\mathfrak{N}$ so that $E(\mathfrak{N})$ is torsion-free by Lemma \ref{lemma:rootofunity}.
Recall that
$E(\mathfrak{N})=F^\times\cap F_{\mathbb{R},+}^\times\cap K(\mathfrak{N})$.
We record the properties of the space $E(\mathfrak{N})\backslash F_{\mathbb{R},+}^\times/U$:

\begin{prop}\label{property:torus}
The followings are true:
\begin{itemize}
\item[(1)] $E(\mathfrak{N})$ freely acts on $F_{\mathbb{R},+}^\times/U$.
\item[(2)] $F_{\mathbb{R},+}^\times/U$ is contractible.
\item[(3)] $E(\mathfrak{N})\backslash F_{\mathbb{R},+}^\times/U$ is homeomorphic to the $r$-torus.
\end{itemize}
\end{prop}

\begin{proof}
(1) and (2) are standard. 
We denote by 
$$
\mathrm{reg}_F:F_{\mathbb{R},+}^\times/U\to\mathbb{R}^{r_1+r_2},\ (x_\tau)_\tau\mapsto
\left(
[F_\tau\colon\mathbb R]\log|x_\tau|
\right)_\tau
$$ 
the regulator map, where $\tau$ runs over the infinite places of $F$. Then, $\mathrm{reg}_F$ is an embedding and $\mathrm{reg}_F E(\mathfrak{N})$ is a full rank lattice of 
$$
\mathbb{R}^{r_1+r_2}_0:=\{r\in\mathbb{R}^{r_1+r_2}\mid\mathrm{Tr}(r)=0\}
$$ 
by the Hilbert product formula and the Dirichlet unit theorem. From this, we observe that $E(\mathfrak{N})\backslash F_{\mathbb{R},+}^\times/U$ is homotopically equivalent to the $r$-torus 
$$
\mathrm{reg}_F E(\mathfrak{N})\backslash \mathbb{R}^{r_1+r_2}_0
$$ 
since the image of $E(\mathfrak{N})\backslash F_{\mathbb{R},+}^\times/U$ under the map is homeomorphic to the product of $\mathrm{reg}_FE(\mathfrak{N})\backslash\mathbb{R}^{r_1+r_2}_0$ and $\mathbb{R}$. 
\end{proof}

Let $\mathbb{A}_F$ be the adele ring of $F$. For $a\in \mathbb{A}_F^\times$, let us denote by $a_\infty$ and $a^{(\infty)}$ the Archimedean and the non-Archimedean part of $a$, respectively.
For $a\in\mathbb{A}_F^{\times}$, let us define a map $\iota_a:E(\mathfrak{N})\backslash F_{\mathbb{R},+}^\times/U\to Y(\mathfrak{N})$ by
$$
\iota_a \colon x\mapsto (a_\infty x,a^{(\infty)}).
$$
The next proposition is standard, which we prove here due to the lack of references.
\begin{prop}\label{property:torus:embedding}
For each $a\in\mathbb{A}_F^{\times}$, the map $\iota_a$ is a homeomorphism onto its image.
\end{prop}

\begin{proof}
Let us choose $\varepsilon\in E(\mathfrak{N})$, $x\in F^\times_{\mathbb{R},+}$ and $u\in U$. Then, 
$$
(a_\infty x,a^{(\infty)})=\big(a_\infty xu,a^{(\infty)}(\varepsilon^{-1})^{(\infty)}\big)=\big(a_\infty\varepsilon_\infty xu,a^{(\infty)}\big)
$$ 
in $Y(\mathfrak{N})$. Hence, the map is well-defined. We first prove its injectivity.
Let us choose $x,y\in F^\times_{\mathbb{R},+}$ such that $(a_\infty x,a^{(\infty)})=(a_\infty y,a^{(\infty)})$ in $Y(\mathfrak{N})$, then 
$$
(a_\infty x,a^{(\infty)})=\big(\beta_\infty a_\infty y,\beta^{(\infty)} a^{(\infty)}k\big)
$$ 
in $(F_\mathbb{R}^\times/U)\times\mathbb{A}_F^{(\infty),\times}$ for some $\beta \in F^\times$ and $k\in K(\mathfrak{N})$. Hence, $\beta_\infty=x^{-1} y$ in $F_\mathbb{R}^\times/U$, which implies that $\beta\in F^\times\cap F_{\mathbb{R},+}^\times$. Also we have $\beta^{(\infty)}=k^{-1}\in K(\mathfrak{N})$. Therefore, $\beta$ is an element of $E(\mathfrak{N})$ so that $x=y$ in $E(\mathfrak{N})\backslash F_{\mathbb{R},+}^\times/U$. On the other hand, $\iota_a$ is open because $\mathbb A_F^{(\infty),\times}/K(\mathfrak N)$ is discrete. Being injetive and open, $\iota_a$ is a homeomorphism onto the image.
%On the other hand, it is a standard fact from class field theory that the image of $\iota_1$ is equal to the identity component of $Y(\mathfrak N)$. This completes the proof.
%\marginpar{proof needs to be completed?}
\end{proof}

Let us denote by $\mathrm{Cl}_F^+(\mathfrak{N})$ the idelic narrow ray class group of $F$ with conductor %\marginpar{with conductor? of conductor?} 
$\mathfrak{N}$, which is given by
$$
\mathrm{Cl}_F^+(\mathfrak{N}):=F^\times\backslash\mathbb{A}_F^\times/\big(F_{\mathbb{R},+}^\times\times K(\mathfrak{N})\big).
$$ 
Let us denote by $Y(\mathfrak{N})^\circ$ the image of $\iota_1$, which is an $r$-torus equipped with the group structure. 
It is a standard fact that $Y(\mathfrak N)$ fibres over $\mathrm{Cl}_F^+(\mathfrak{N})$. We state and prove it, due to the lack of suitable references.

\begin{prop}\label{idelic:torus:decomposition}
The map
$$
Y(\mathfrak{N})\rightarrow\mathrm{Cl}_F^+(\mathfrak{N}),\ (x,a)\mapsto ax
$$
is a principal $Y(\mathfrak{N})^\circ$-bundle, where the space $\mathrm{Cl}_F^+(\mathfrak{N})$ is equipped with the discrete topology.
Furthermore, the fibre of the bundle over an ideal class $[a]\in\mathrm{Cl}_F^+(\mathfrak{N})$ is given by the image of the map $\iota_a$.%\marginpar{basically trivial?}
\end{prop}

\begin{proof}
Let us choose $(x,a),(y,b)\in F_\mathbb{R}^\times\times\mathbb{A}_F^{(\infty),\times}$ such that $(x,a)=(y,b)$ in $Y(\mathfrak{N})$. Then, $x=\beta_\infty y$ in $F^\times_{\mathbb{R}}/U$ and $a=\beta^{(\infty)} b$ in $\mathbb{A}_F^{(\infty),\times}/K(\mathfrak{N})$ for some $\beta \in F^\times$. 
Thus, 
$$
ax=\beta by=by
$$
in $\mathrm{Cl}_F^+(\mathfrak{N})$, so the map is well-defined. Surjectivity and continuity are clear.

Let us choose $(x,c)$ in the fibre over $[a]\in \mathrm{Cl}_F^+(\mathfrak{N})$, then, $ay=\beta ckx$ for some $\beta\in F^\times$, $y\in F_{\mathbb{R},+}^\times$, and $k\in K(\mathfrak{N})$. Therefore,
$$
\iota_{a}(y)=(a_\infty y,a^{(\infty)})=(\beta_\infty x,\beta^{(\infty)}ck)=(x,c).
$$
Conversely, let us choose $(a_\infty x,a^{(\infty)})$ in the image of $\iota_a$, then $(a_\infty x,a^{(\infty)})$ goes to $ax$ under the map $Y(\mathfrak{N})\rightarrow\mathrm{Cl}_F^+(\mathfrak{N})$ and $ax=a$ in $\mathrm{Cl}_F^+(\mathfrak{N})$, which implies that $(a_\infty x,a^{(\infty)})$ is in the fibre over $[a]$. 
We can easily check that $Y(\mathfrak{N})^\circ$ simply transitively acts on each fibre.
So we are done.
\end{proof}

\begin{remark}\label{representative:choice} From {\rm Proposition \ref{idelic:torus:decomposition}}, we observe that
$Y(\mathfrak{N})$ is homeomorphic to the $\mathrm{Cl}_F^+(\mathfrak{N})$-copy of $r$-torus $E(\mathfrak{N})\backslash F^\times_{\mathbb{R},+}/U$ under the map $\bigsqcup_{[a]\in\mathrm{Cl}_F^+(\mathfrak{N})}\iota_a$.
We can easily check that the map $\iota_a$ depends on $a$, but the image of $\iota_a$ is independent on the choice of a representative of a ray class $[a]\in\mathrm{Cl}_F^+(\mathfrak{N})$.
\end{remark}

%Let $R$ be a field. Note that $E(\mathfrak{N})$ is a free abelian group of rank $r:=r_1+r_2-1$ by Dirichlet's unit theorem. Let us identify $\mathbb{Z}^r\cong E(\mathfrak{N})$, $(n_i)_i\mapsto \varepsilon_1^{n_1},\cdots,\varepsilon_{r}^{n_r}$, where $\{\varepsilon_i\}_i$ is a complete set of generators of $E(\mathfrak{N})$. Thus, from the K{\" u}nneth formula, we obtain that $$H^k(Y(\mathfrak{N}),R)\cong\prod_{a\in\pi_0(Y(\mathfrak{N}))}\bigoplus_{j_1+\cdots+j_r=k}\bigotimes_{i=1}^r H^{j_i}(\mathbb{Z},R).$$

\section{Derived Hecke algebra of torus over number field}\label{section:hecke}
In this section, let us define the derived Hecke algebra and its action on the cohomology of torus $Y(\mathfrak{N})$. To do so, we follow the explicit description by Venkatesh \cite[\textsection 2]{venkatesh2016derived}.

Let $R$ be a ring.
For a prime $v$ of $F$, let $F_v$ be the $v$-adic completion of $F$, $O_v$ the valuation ring on $F_v$, $N(v)$ the absolute norm of $v$, and $\kappa_v$ the residue field of $v$. 

\subsection{Derived Hecke operator}
Let $v$ be a prime of $F$.
For $z\in F_v^\times/O_v^\times$, let $z:Y(\mathfrak{N})\rightarrow Y(\mathfrak{N})$ be the map induced by the multiplication $a\mapsto az$ with $z$. From this, we obtain the pullback map 
$$
z^*:H^\bullet(Y(\mathfrak{N}),R)\rightarrow H^\bullet(Y(\mathfrak{N}),R).
$$ 
Let $v$ be a prime ideal of $F$ coprime to $\mathfrak{N}$ such that $N(v)$ is invertible in $R$.
We have the following canonical projection:
$$
\pi_{v,\mathfrak{N}}:Y(v\mathfrak{N})\rightarrow Y(\mathfrak{N}),
$$ 
which is a principal $\kappa_v^\times$-bundle, where the action of $k\in\kappa_v^\times$ on $(a,x)\in Y(v\mathfrak{N})$ is given by $k\cdot(a,x)=(ak^{-1},x)$. 
For a discrete group $G$, let $\mathcal{B}G$ be a classifying space of $G$.
By its universal property, $\pi_{v,\mathfrak{N}}$ corresponds to a map 
\begin{align}\label{eq:3.1}
\phi_{v,\mathfrak{N}}:Y(\mathfrak{N})\rightarrow\mathcal{B}\kappa_v^\times,
\end{align}
which is determined  uniquely up to homotopy.
%Note that we always take the model of $\mathcal{B}G$ given by the bar construction.
%\textcolor{red}{Then, there is a map $\phi_{v,\mathfrak{N}}:Y(\mathfrak{N})\rightarrow\mathcal{B}\kappa_v^\times$ of topological spaces, which is a representative of the homotopy class of maps corresponding to $\pi_{v,\mathfrak{N}}$.}
Note that the inflation map $H^\bullet(\kappa_v^\times,R)\rightarrow H^\bullet(O_v^\times,R)$ is clearly an isomorphism as $N(v)$ is invertible in $R$. Thus, from now on, for $\alpha\in H^\bullet(O_v^\times,R)$, we denote by $\alpha\in H^\bullet(\kappa_v^\times,R)$ the image of $\alpha\in H^\bullet(O_v^\times,R)$ under the inverse of the inflation map by abusing the notation.
Let us denote by $\langle\cdot\rangle$ the composition of the following maps:
\begin{equation}\label{pullback:homotopy}
\begin{tikzcd}
H^\bullet(O_v^\times,R) \arrow{r}{\cong} &  H^\bullet(\kappa_v^\times,R) \arrow{r}{\cong} &  H^\bullet(\mathcal{B}\kappa_v^\times,R)\arrow{r}{\phi_{v,\mathfrak{N}}^*} & H^\bullet(Y(\mathfrak{N}),R).
\end{tikzcd}
\end{equation}
Note that $\langle\cdot\rangle:H^\bullet(O_v^\times,R)\to H^\bullet(Y(\mathfrak{N}),R)$ is a map of graded $R$-algebras since $H^\bullet(-,R)$ is a functor into the category of graded $R$-algebras.

\begin{defn}\label{derived:hecke:operator:definition}
For $z\in F_v^\times/O_v^\times$ and $\alpha\in H^\bullet(O_v^\times,R)$, let us define $h_{z,\alpha}\in\mathrm{End}_RH^\bullet(Y(\mathfrak{N}),R)$ as follows:
$$
h_{z,\alpha}c:=\langle\alpha\rangle\cup z^*c\text{ for }c\in H^\bullet(Y(\mathfrak{N}),R),
$$ 
where $\cup$ is the cup product on the graded $R$-algebra $H^\bullet(Y(\mathfrak{N}),R)$. 
\end{defn}
$h_{z,\alpha}$ is called {\it derived Hecke operator}.

\begin{remark}\label{level:dividing:trivial}
We can also define the derived Hecke action for prime ideals $v$ of $F$ dividing $\mathfrak{N}$, but in the case of torus, such action is trivial for positive degree.
\end{remark}

Let us record the properties of the derived Hecke operators:
\begin{prop}\label{property:hecke:operator}
\begin{itemize}
\item[(1)] The map $F_v^\times/O_v^\times\to\mathrm{Aut}_RH^\bullet(Y(\mathfrak{N}),R)$ defined by $z\mapsto h_{z,1}$ is a group homomorphism.
\item[(2)] The map $H^\bullet(O_v^\times,R)\to\mathrm{End}_RH^\bullet(Y(\mathfrak{N}),R)$ defined by $\alpha\mapsto h_{1,\alpha}$ is a map of graded $R$-algebras.
\end{itemize}
\end{prop}

\begin{proof} 
(1) is immediate from the definition. 
(2) is immediate since $\langle\cdot\rangle$ is a map of graded $R$-algebras.
\end{proof}

%\begin{remark}
%From Proposition \ref{idelic:torus:decomposition}, we observe that the principal $\kappa_v^\times$-bundle 
%$$
%Y(v\mathfrak{N})\rightarrow Y(\mathfrak{N})
%$$ 
%is determined by the canonical projection $\mathrm{Cl}_F^+(v\mathfrak{N})\rightarrow \mathrm{Cl}_F^+(\mathfrak{N})$ and the inclusion $E(v\mathfrak{N})\rightarrow E(\mathfrak{N})$. By applying the snake lemma on the fundamental sequences of ray class groups, we obtain the following isomorphisms:
%$$
%\mathrm{ker}\bigg(\frac{(O_F/v\mathfrak{N})^\times}{E/E_{v\mathfrak{N},1}}\rightarrow \frac{(O_F/\mathfrak{N})^\times}{E/E_{\mathfrak{N},1}}\bigg)\cong\mathrm{ker}(\mathrm{Cl}_F^+(v\mathfrak{N})\rightarrow \mathrm{Cl}_F^+(\mathfrak{N})),\ \alpha\mapsto\alpha O_F.
%$$
%Applying the snake lemma again, we obtain
%$$
%\frac{\kappa_v^\times}{E(\mathfrak{N})/E(v\mathfrak{N})}\cong\frac{\ker\big((O_F/v\mathfrak{N})^\times\rightarrow (O_F/\mathfrak{N})^\times\big)}{E(\mathfrak{N})/E(v\mathfrak{N})}\cong\ker\big(\mathrm{Cl}_F^+(v\mathfrak{N})\rightarrow \mathrm{Cl}_F^+(\mathfrak{N})\big),
%$$
%where the last map is given by $\alpha\mapsto\alpha O_F$ and the inclusion $E(\mathfrak{N})/E(v\mathfrak{N})\rightarrow \kappa_v^\times$ is induced from the embedding $E(\mathfrak{N})\rightarrow O_v^\times$.
%In conclude, the bundle $Y(v\mathfrak{N})\rightarrow Y(\mathfrak{N})$ is determined by the embedding $E(\mathfrak{N})\rightarrow O_v^\times$.
%\end{remark}
Denote by $i_{v,\mathfrak{N}}$ the injection
$$
i_{v,\mathfrak{N}}:E(\mathfrak{N})/E(v\mathfrak{N})\to\kappa_v^\times, \varepsilon\mapsto\varepsilon_v\text{ mod }v
$$
induced by the embedding $O_F^\times\to O_v^\times,\varepsilon\mapsto\varepsilon_v$.
There is an action $\varepsilon\in E(\mathfrak{N})/E(v\mathfrak{N})$ on $x\in E(\mathfrak{N})\backslash F_{\mathbb{R},+}^\times/U$ given as $\varepsilon\cdot x:=\varepsilon x$.
We can check that this action makes the canonical projection 
\begin{align}\label{eq:3.3}
E(v\mathfrak{N})\backslash F_{\mathbb{R},+}^\times/U\to E(\mathfrak{N})\backslash F_{\mathbb{R},+}^\times/U
\end{align}
principal $E(\mathfrak{N})/E(v\mathfrak{N})$-bundle. 
Let $a$ be an element of $\mathbb{A}_F^\times$ and recall that $\iota_{a}:E(v\mathfrak{N})\backslash F_{\mathbb{R},+}^\times/U\to Y(v\mathfrak{N})$ is the map defined by $\iota_{a}(x)=(a_\infty x,a^{(\infty)})$.
Then, we have the following diagram:
\begin{equation}\label{maps:bundles}
\begin{tikzcd}
E(v\mathfrak{N})\backslash F_{\mathbb{R},+}^\times/U \arrow{r}{\iota_a}\arrow{d} & Y(v\mathfrak{N}) \arrow{d}{\pi_{v,\mathfrak{N}}}  
\\
E(\mathfrak{N})\backslash F_{\mathbb{R},+}^\times/U \arrow{r}{\iota_a} & Y(\mathfrak{N}).
\end{tikzcd}
\end{equation}
In the way $\phi_{v,\mathfrak N}$ of \eqref{eq:3.1} arises from $\pi_{v,\mathfrak N}$, \eqref{eq:3.3} gives rise to a map $\psi_{v,\mathfrak N}$ from $E(\mathfrak{N})\backslash F_{\mathbb{R},+}^\times/U$ to $\mathcal{B}\frac{E(\mathfrak{N})}{E(v\mathfrak{N})}$.
Combining them, we obtain the following diagram:
\begin{equation}\label{maps:spaces}
\begin{tikzcd}
E(\mathfrak{N})\backslash F_{\mathbb{R},+}^\times/U \arrow{r}{\iota_a} \arrow{d}{\psi_{v,\mathfrak{N}} } & Y(\mathfrak{N}) \arrow{d}{\phi_{v,\mathfrak{N}}}
\\
\mathcal{B}\frac{E(\mathfrak{N})}{E(v\mathfrak{N})} \arrow{r}{\mathcal{B}i_{v,\mathfrak{N}}} & \mathcal{B}\kappa_v^\times.
\end{tikzcd}
\end{equation}
%\textcolor{red}{where $\psi_{v,\mathfrak{N}}$ is a map of topological spaces, which is a representative of the homotopy class of the maps corresponding to the principal $E(\mathfrak{N})/E(v\mathfrak{N})$-bundle.}
We claim that \eqref{maps:spaces} is commutative up to homotopy. 
Note that
$$
\iota_{a}(\varepsilon\cdot x)=(a_\infty\varepsilon x,a^{(\infty)})=(a_\infty x,i_{v,\mathfrak{N}}(\varepsilon^{-1}) a^{(\infty)})
$$
in $Y(v\mathfrak{N})$ for $\varepsilon\in E(\mathfrak{N})$. 
Therefore, the diagram (\ref{maps:bundles}) is equivariant with respect to the map $i_{v,\mathfrak{N}}$, which implies that the diagram (\ref{maps:spaces}) is homotopy-commutative.
Taking the cohomology ring functor to the diagram (\ref{maps:spaces}), we obtain the following commutative diagram of cohomology groups:
\begin{equation}\label{cohomology:commutative:diagram:2}
\begin{tikzcd}
H^\bullet(E(\mathfrak{N})\backslash F_{\mathbb{R},+}^\times/U,R)  &  \arrow{l}{\iota_a^*}  H^\bullet(Y(\mathfrak{N}),R) 
\\
H^\bullet(\mathcal{B}\frac{E(\mathfrak{N})}{E(v\mathfrak{N})},R) \arrow{u}{\psi_{v,\mathfrak{N}}^*} & \arrow{l}{\mathcal{B}i_{v,\mathfrak{N}}^*}  H^\bullet(\mathcal{B}\kappa_v^\times,R) \arrow{u}{\phi_{v,\mathfrak{N}}^*}.
\end{tikzcd}
\end{equation}

On the other hand, we have the following commutative diagram:
$$
\begin{tikzcd}
F_{\mathbb{R},+}^\times/U \arrow{d} \arrow{r} &  E(v\mathfrak{N})\backslash F_{\mathbb{R},+}^\times/U \arrow{d} 
\\
E(\mathfrak{N})\backslash F_{\mathbb{R},+}^\times/U \arrow{r}{=} &  E(\mathfrak{N})\backslash F_{\mathbb{R},+}^\times/U,
\end{tikzcd}
$$
where the left vertical map and the top horizontal map are naturally principal $E(\mathfrak{N})$ bundle and canonical principal $E(v\mathfrak{N})$ bundle, respectively, where $\varepsilon\in E(\mathfrak{N})$ acts on $x\in F_{\mathbb{R},+}^\times/U$ as $\varepsilon\cdot x:=\varepsilon x$. Therefore, this diagram is equivariant under the canonical projection $E(\mathfrak{N})\to E(\mathfrak{N})/E(v\mathfrak{N})$.
Hence, we obtain the following homotopy-commutative diagram
$$
\begin{tikzcd}
E(\mathfrak{N})\backslash  F_{\mathbb{R},+}^\times/U \arrow{d} \arrow{r}{=} & E(\mathfrak{N})\backslash F_{\mathbb{R},+}^\times/U \arrow{d}{\psi_{v,\mathfrak{N}}} 
\\
\mathcal{B}E(\mathfrak{N}) \arrow{r} & \mathcal{B}\frac{E(\mathfrak{N})}{E(v\mathfrak{N})}
\end{tikzcd}
$$
where the bottom horizontal map and the left vertical map are induced by the canonical projection and the canonical principal $E(\mathfrak{N})$-bundle, respectively.
Taking the functor $H^\bullet(-,R)$ to the just above diagram, we obtain the following commutative diagram:
\begin{equation}\label{cohomology:commutative:diagram:3}
\begin{tikzcd}
& H^\bullet(E(\mathfrak{N})\backslash F_{\mathbb{R},+}^\times/U,R)
\\
H^\bullet(\mathcal{B}E(\mathfrak{N}),R)  \arrow{ur}{\cong} & H^\bullet(\mathcal{B}\frac{E(\mathfrak{N})}{E(v\mathfrak{N})},R) \arrow{l}\arrow{u}{\psi^*_{v,\mathfrak{N}}},
\end{tikzcd}
\end{equation}
where the diagonal arrow is an isomorphism of graded $R$-algebras by Proposition \ref{property:torus} and the horizontal arrow is the functorial map induced from the canonical projection of groups.
Combining the diagrams (\ref{cohomology:commutative:diagram:2}) and (\ref{cohomology:commutative:diagram:3}), we obtain the following commutative diagram:
\begin{equation}\label{cohomology:commutative:diagram:4}
\begin{tikzcd}
 & H^\bullet(E(\mathfrak{N})\backslash F_{\mathbb{R},+}^\times/U,R)  &  \arrow{l}{\iota_a^*}  H^\bullet(Y(\mathfrak{N}),R) 
\\
H^\bullet(\mathcal{B}E(\mathfrak{N}),R) \arrow{ur}{\cong}  & H^\bullet(\mathcal{B}\frac{E(\mathfrak{N})}{E(v\mathfrak{N})},R) \arrow{u}{\psi_{v,\mathfrak{N}}^*}\arrow{l} & \arrow{l}{\mathcal{B}i_{v,\mathfrak{N}}^*}  H^\bullet(\mathcal{B}\kappa_v^\times,R) \arrow{u}{\phi_{v,\mathfrak{N}}^*}.
\end{tikzcd}
\end{equation} 
Due to the property of the functor $\mathcal{B}$, we finally obtain the following commutative diagram from \eqref{cohomology:commutative:diagram:4}:
\begin{equation}\label{cohomology:commutative:diagram}
\begin{tikzcd}
 H^\bullet(E(\mathfrak{N})\backslash F_{\mathbb{R},+}^\times/U,R)  &  \arrow{l}{\iota_a^*}  H^\bullet(Y(\mathfrak{N}),R) 
\\
H^\bullet(E(\mathfrak{N}),R) \arrow{u}{\cong} &  \arrow{l}{i_{v,\mathfrak{N}}^*}  H^\bullet(\kappa_v^\times,R) \arrow{u}{\phi_{v,\mathfrak{N}}^*},
\end{tikzcd}
\end{equation}
where we denote by $i_{v,\mathfrak{N}}:E(\mathfrak{N})\to\kappa_v^\times$ the composition of the canonical projection $E(\mathfrak{N})\to E(\mathfrak{N})/E(v\mathfrak{N})$ with the injection $i_{v,\mathfrak{N}}$ by abusing the notation.
From now on, by abusing the notation, let us denote by $\iota^*_a$ the composition of the identification $H^\bullet(E(\mathfrak{N}),R)\xrightarrow{\cong} H^\bullet(E(\mathfrak{N})\backslash F_{\mathbb{R},+}^\times/U,R)$ and the map $\iota^*_a$.
Then, we have the following explicit description on $h_{z,\alpha}$:
\begin{prop}\label{computation:hecke:operator}  
Let $z\in F_v^\times/O_v^\times$ and $\alpha\in H^\bullet(O_v^\times,R)$.
Then,
$$
\iota_a^*(h_{z,\alpha}c)=i_{v,\mathfrak{N}}^*(\alpha)\cup\iota_{za}^*(c)\text{ for }c\in H^\bullet(Y(\mathfrak{N}),R).
$$
\end{prop}

\begin{proof}
From the commutative diagram (\ref{cohomology:commutative:diagram}), we obtain that $\iota_a^*(\langle\alpha\rangle)=i_{v,\mathfrak{N}}^*(\alpha)$ under the identification. 
Also we have that $\iota_a^*(z^*c)=(z\circ\iota_{a})^*(c)=\iota_{za}^*(c)$. 
Then, we are done since $\iota_a^*$ is a map of cohomology rings.
\end{proof}

\subsection{Derived Hecke algebra}
Let $v$ be a prime ideal of $F$ whose absolute norm is invertible in $R$.
Let us recall that the abstract derived Hecke algebra at $v$ with coefficients in $R$ is given by 
$$
\mathcal{H}_{v,R}^\bullet=\mathrm{Ext}^\bullet_{C_{\rm c}(F_v^\times,R)}(R[F_v^\times/K(\mathfrak{N})_v],R[F_v^\times/K(\mathfrak{N})_v]).
$$
There is an explicit description on the action of abstract derived Hecke algebras $\mathcal{H}^\bullet_{v,R}$ on the cohomology ring $H^\bullet(Y(\mathfrak{N}),R)$. 
We can easily check that $\mathcal{H}_{v,R}^j$ acts trivially on $H^\bullet(Y(\mathfrak{N}),R)$ if $j>0$ and $v\mid \mathfrak{N}$, as mentioned in Remark \ref{level:dividing:trivial}. Therefore, let us assume that $v$ is coprime to $\mathfrak{N}$.
By Venkatesh \cite[\textsection 2.4, (25)]{venkatesh2016derived}, we have the following isomorphism of $R$-modules:
\begin{equation}\label{explicit:description}
R[F_v^\times/O_v^\times]\otimes_R H^\bullet(O_v^\times,R)\cong \mathcal{H}^\bullet_{v,R}.
\end{equation}
Let us define an action of $\mathcal{H}^\bullet_{v,R}$ on $c\in H^\bullet(Y(\mathfrak{N}),R)$ by $(z,\alpha)\cdot c:= h_{z,\alpha}c$, where $(z,\alpha)\in\mathcal{H}^\bullet_{v,R}$ is the image of $z\otimes\alpha\in R[F_v^\times/O_v^\times]\otimes_R H^\bullet(O_v^\times,R)$ under the isomorphism (\ref{explicit:description}).
Let us write the representation of the given action as follows: 
$$
\mathcal{H}^\bullet_{v,R}\to\mathrm{End}_R H^\bullet(Y(\mathfrak{N}),R),\ (z,\alpha)\mapsto h_{z,\alpha},
$$
which is clearly a map of graded $R$-algebras.
Let us denote by $\mathbb{T}_{v,R}^\bullet$ the image of $\mathcal{H}_{v,R}^\bullet$ in $\mathrm{End}_{R}H^\bullet(Y(\mathfrak{N}),R)$, which is a graded $R$-algebra.
Note that any pair of operators in $\mathbb{T}_{v,R}^\bullet$ for any distincts pair of primes $v$ commute in $\mathrm{End}_{R}H^\bullet(Y(\mathfrak{N}),R)$ by Venkatesh \cite[\textsection 2.10, Remark]{venkatesh2016derived}. 
From this, we may define the algebra generated by $\mathbb{T}_{v,R}^\bullet$ for primes $v$ of $F$. 
Let $T_0$ be the set of primes $v$ of $F$ such that $N(v)$ is invertible in $R$.

\begin{defn}[Derived Hecke algebra, {Venkatesh \cite[\textsection 2]{venkatesh2016derived}}]\label{definition:derived:hecke:algebra}
Let us denote by $\mathbb{T}_R^\bullet$ the graded $R$-subalgebra of $\mathrm{End}_{R}H^\bullet(Y(\mathfrak{N}),R)$ generated by $\mathbb{T}_{v,R}^\bullet$ for all $v\in T_0$.
We call $\mathbb{T}_R^\bullet$ {\it derived Hecke algebra} of $Y(\mathfrak{N})$ with coefficients in $R$.
%For $j\geq 0$, let us denote by $\mathbb{T}_R^j$ the $R$-submodule of degree $j$ maps in $\mathbb{T}_R^\bullet$.
\end{defn} 

As stated in Feng-Harris \cite[\textsection 6.1.3]{feng2024derived}, the convolution product of the left hand side of the map {\rm (\ref{explicit:description})} corresponds to the multiplication in $\mathrm{Ext}^\bullet$ groups on the right hand side. Therefore, $\mathcal{H}^\bullet_{v,R}$ is graded commutative. This subsumes the graded commutativity of $\mathbb{T}_R^\bullet$. In the next proposition, we provide a direct proof of it.

\begin{prop}\label{hecke:algebra:commutative}
\begin{itemize}
\item[(1)] $\mathbb{T}_{R}^0$ is non-zero.
\item[(2)] $\mathbb{T}_R^\bullet$ is a graded commutative $\mathbb{T}_{R}^0$-algebra. 
\end{itemize}
\end{prop}

\begin{proof}
From Proposition \ref{property:hecke:operator} (1), we obtain the non-triviality of $\mathbb{T}^0_R$ since $h_{1,1}\in\mathbb{T}_{R}^0$ is the identity operator on $H^\bullet(Y(\mathfrak{N}),R)$, where the second subindex $1$ of $h_{1,1}$ is the identity element of $H^0(O_v^\times,R)\cong R$. %Note that the operators in $\mathbb{T}^\bullet_R$ and those of $\mathbb{T}^0_R$ commute by Proposition \ref{property:hecke:operator} (1). 

By the definition, $\mathbb{T}_{R}^{i}$'s are $\mathbb{T}_{R}^0$-modules and $\mathbb{T}_{R}^{i}\mathbb{T}_{R}^j\subset\mathbb{T}_{R}^{i+j}$ for any $i,j\geq 0$. 
Due to Venkatesh \cite[\textsection 2.10, Remark]{venkatesh2016derived} we know that the derived Hecke operators at different primes commute. Therefore, it is enough to show that
$$
h_{z_1,\alpha_1}h_{z_2,\alpha_2}=h_{z_2,\alpha_2}h_{z_1,\alpha_1}
$$
for $z_1,z_2\in F_v^\times/O_v^\times$ and $\alpha_1,\alpha_2\in H^\bullet(O_v^\times,R)$.
Let us choose $a\in\mathbb{A}_F^\times$ and $c\in H^\bullet(Y(\mathfrak{N}),R)$. Note that $\iota_a^*$ and $z_i^*$ are maps of cohomology rings. Thus,
$$
\iota_a^*(h_{z_1,\alpha_1}h_{z_2,\alpha_2}c)=\iota_a^*\langle\alpha_1\rangle\cup\iota_a^*z_1^*(\langle\alpha_2\rangle\cup z_2^*c)=i_{v,\mathfrak{N}}^*(\alpha_1)\cup i_{v,\mathfrak{N}}^*(\alpha_2)\cup\iota_{z_2z_1a}^*c
$$
by the commutative diagram (\ref{cohomology:commutative:diagram}). Note that $\cup$ is graded commutative and $z_1z_2=z_2z_1$. Since $a$ and $c$ are arbitrary, this completes the proof.
\end{proof}

Let $a\in\mathbb{A}_F^\times$. The map $\iota_a:E(\mathfrak{N})\backslash F_{\mathbb{R},+}^\times/U\to Y(\mathfrak{N})$ is a homeomorphism onto the image, so it is proper. Therefore, it gives rise to the pushforward map 
$$
\iota_{a,*}:
%H^\bullet(E(\mathfrak{N}),R)\cong 
H^\bullet(E(\mathfrak{N})\backslash F_{\mathbb{R},+}^\times/U,R)\to H^\bullet(Y(\mathfrak{N}),R).
$$
%on the cohomology groups. 
For $a\in\mathbb{A}_F^\times$, we put %let us denote by 
\begin{align}\label{eq:1a}
1_a:=\iota_{a,*}(1)\in H^0(Y(\mathfrak{N}),R).
\end{align}
Using the singular cochain complex, we can represents $1_a$ as the characteristic function on $Y(\mathfrak{N})$ supported on $\mathrm{Im}(\iota_a)$. Thus, by Remark \ref{representative:choice}, $1_a$ depends only on the ray class of modulus $\mathfrak{N}$ represented by $a$.
We may choose a complete representative set 
$$
C(\mathfrak{N})\subset\mathbb{A}_F^{(\infty),\times}
$$ 
of $\mathrm{Cl}_F^+(\mathfrak{N})$ such that $1\in C(\mathfrak{N})$, where $1$ is the identity element of $\mathbb{A}_F^{(\infty),\times}$. 
By abusing the notation, for $a \in C(\mathfrak N)$ we let $1_a$ stand for $1_{(a,1)}$ of \eqref{eq:1a}.
% for $a\in C(\mathfrak{N})$.
Then, $1_a$'s generate $H^0(Y(\mathfrak{N}),R)$ for $a\in C(\mathfrak{N})$.
Under the isomorphism $H^0(E(\mathfrak{N}),R)\cong R$, $\iota_b^*(1_a)=\delta_{ab}$ for $a,b\in C(\mathfrak{N})$.

Let $a\in C(\mathfrak{N})$ and $z\in\mathbb{A}_F^\times/\widehat{O}_F^\times$. 
Let $z(a)$ be the element of $C(\mathfrak{N})$ such that $z(a)=(\prod_{v\in T_0}z_v)a$ in $\mathrm{Cl}_F^+(\mathfrak{N})$. 
%\marginpar{why $T_0$?}
Since the product is finite as $h_{1,1}$ is the identity map, the operator $a\mapsto z(a)$ on $C(\mathfrak{N})$ is well-defined. From this, we obtain the following permutation representation:
$$
\mathbb{A}_F^\times/\widehat{O}_F^\times\rightarrow\mathrm{Perm}\big(C(\mathfrak{N})\big),\ z\mapsto \big(a\mapsto z(a)\big).
$$
Let us denote
$$
h_{z,1}:=\prod_{v\in T_0}h_{z_v,1}\in\mathbb{T}^0_R,
$$ 
which is well defined by Proposition \ref{hecke:algebra:commutative}.
Based on the fact that $1_a$ is a characteristic function, we obtain that
\begin{equation}\label{hecke:action:representation}
h_{z,1}1_a=z^*1_a=1_{z^{-1}(a)}.
\end{equation}
Let us compute the $\mathbb{F}_p$-vector space structure of $\mathbb{T}^0_R$ and the $\mathbb{T}^0_R$-module structure of $H^0(Y(\mathfrak{N}),R)$ and $\mathbb{T}^1_R$:

\begin{prop}\label{Hecke:module:structure}

\begin{itemize}
\item[(1)] $H^0(Y(\mathfrak{N}),R)=\mathbb{T}_{R}^01_1$.
\item[(2)] $\mathbb{T}^0_R=\sum_{z\in \mathbb{A}_F^\times/\widehat{O}_F^\times}\mathbb{F}_ph_{z,1}$.
\item[(3)] 
$
\mathbb{T}_{R}^1=\sum_{\alpha\in\bigsqcup_{v\in T_0}\mathrm{Hom}(O_v^\times,R)}\mathbb{T}_{R}^0h_{1,\alpha}.
$
\end{itemize}
\end{prop}

\begin{proof}
The claim (1) is immediate from the definition of $\mathbb{T}^0$ and \eqref{hecke:action:representation}.
The claims (2) and (3) follow from the definition of $\mathbb{T}_R^\bullet$, Proposition \ref{property:hecke:operator} and Proposition \ref{hecke:algebra:commutative}.
\end{proof}

\section{Non-vanishing mod {\it p} of derived Hecke algebra}\label{section:nonvanishing}
Let $p$ be a rational prime. We take $R$ to be $\mathbb{F}_p$, the finite field with $p$ elements.
%Let us set $R=\mathbb{F}_p$.
Then, $T_0$ becomes the set of primes $v$ of $F$ coprime to $p\mathfrak{N}$.
Define $T_1\subset T_0$ to be the set of primes $v$ in $T_0$ such that $p\mid N(v)-1$.
We can observe from the Chebotarev density theorem that $T_1$ has a positive proportion in the set of primes of $F$. 
Note that $H^j(O_v^\times,\mathbb{F}_p)\cong H^j(\kappa_v^\times,\mathbb{F}_p)$ if $v\in T_0$. 
%Also note that $H^j(O_v^\times,\mathbb{F}_p)$ for $j>0$ is non-vanishing and isomorphic to $\mathbb{F}_p$ if and only if $v\in T_1$, due to the K{\"u}nneth formula. 
For $j\geq 0$, we put
$$
\mathbb{T}^j:=\mathbb{T}_{\mathbb{F}_p}^j\subset\mathrm{End}_{\mathbb{F}_p}H^\bullet(Y(\mathfrak{N}),\mathbb{F}_p).
$$ 
%so that $\mathbb{T}^j$ is a $\mathbb{F}_p$-vector space.
%Derived Hecke algebras of degree $0$, $1$ and $2$ generate the rest:

\begin{prop}\label{property:hecke:operator:2}
$\mathbb{T}^\bullet$ is generated by $\mathbb{T}^0$, $\mathbb{T}^1$, and $\mathbb{T}^2$ as a ring.
\end{prop}

\begin{proof} 
Let $v\in T_1$. Let us recall that $H^j(O_v^\times,\mathbb{F}_p)\cong H^j(\kappa_v^\times,\mathbb{F}_p)\cong \mathbb{F}_p$.  
From the periodicity (see Artin-Tate \cite[Chapter Preliminary, \textsection\,2, Theorem B]{artin2008class}) and the graded commutativity of the cup product on the cohomology ring of cyclic groups, the ring structure of $H^\bullet(O_v^\times,\mathbb{F}_p)$ is given by 
$$
H^\bullet(O_v^\times,\mathbb{F}_p)\cong H^\bullet(\kappa_v^\times,\mathbb{F}_p)\cong\mathbb{F}_p[\alpha,\beta]/(\alpha^2),
$$
where $\alpha$ and $\beta$ map to generators of $H^1(O_v^\times,\mathbb{F}_p)$ and $H^2(O_v^\times,\mathbb{F}_p)$, respectively.
Consequently, by definition and the isomorphism \eqref{explicit:description}, $\mathbb{T}_v^0$, $\mathbb{T}_v^1$, and $\mathbb{T}_v^2$ generate $\mathbb{T}_v^\bullet$ as a graded $R$-algebra. Recalling $\mathbb{T}^\bullet$ is generated by $\mathbb{T}_v^\bullet$ for $v\in T_0$, the proof is completed.
\end{proof}

Recall the map $\Psi$ introduced in \eqref{eq:psi}.
By Proposition \ref{hecke:algebra:commutative} (2), $\Psi$ is $\mathbb{T}^0$-linear.
Now we prove that $\Psi$ is a $\mathbb{T}^0$-linear isomorphism.
Define $t_p\geq 0$ to be the dimension of the following subspace
$$
\sum_{v\in T_1}i_{v,\mathfrak{N}}^*\mathrm{Hom}(\kappa_v^\times,\mathbb{F}_p)\subset\mathrm{Hom}(E(\mathfrak{N}),\mathbb{F}_p).
$$

%Since the $\mathbb F_p$-dimension of $\mathrm{Hom}(E(\mathfrak{N}),\mathbb{F}_p)$ is $r$, we have $t\leq r$. 
Let $h_F^+(\mathfrak{N})$ be the cardinality of $\mathrm{Cl}_F^+(\mathfrak{N})$.
%Then, the dimension of both the image and the domain of $\Psi$ can be written in terms of $h_F^+(\mathfrak{N})$ and $t$:

\begin{prop}\label{dimension:prop:2}
The image of $\Psi$ has dimension $h_F^+(\mathfrak{N})\cdot t_p$.
\end{prop}

\begin{proof} 
By Proposition \ref{Hecke:module:structure} (1), the image $\mathrm{Im}(\Psi)$ coincides with $\mathbb{T}^1 1_1$. By Proposition \ref{property:hecke:operator}, Proposition \ref{hecke:algebra:commutative}, and Proposition \ref{Hecke:module:structure} (2), (3),
$$
\mathbb{T}^1 1_1=\sum_{\alpha\in\bigsqcup_{v\in T_0}\mathrm{Hom}(E(\mathfrak{N}),\mathbb{F}_p)}\mathbb{T}^0h_{1,\alpha} 1_1=\sum_{a\in C(\mathfrak{N})}\sum_{\alpha}\mathbb{F}_ph_{1,\alpha} 1_a.
$$
The spaces $\mathbb{F}_ph_{1,\alpha} 1_a$ are linearly disjoint for $a\in C(\mathfrak{N})$ as $h_{1,\alpha} 1_a$'s are cohomology classes supported on the distinct connected components. Therefore, we obtain from the above equation that 
$$
\mathrm{dim}_{\mathbb{F}_p}\mathbb{T}^1 1_1=h_F^+(\mathfrak{N})\cdot \mathrm{dim}_{\mathbb{F}_p}\sum_{\alpha}\mathbb{F}_ph_{1,\alpha} 1_1.
$$
Under the isomorphism $\prod_{b\in C(\mathfrak{N})}\iota_b^*$, the image of $\sum_{\alpha}\mathbb{F}_ph_{1,\alpha} 1_1$ is isomorphic to
$$
\sum_{\alpha}\mathbb{F}_p i_{v_\alpha,\mathfrak{N}}^*(\alpha)\subset\mathrm{Hom}(E(\mathfrak{N}),\mathbb{F}_p)
$$ 
due to Proposition \ref{computation:hecke:operator}, where $v_\alpha$ is the prime ideal of $F$ such that $\alpha\in\mathrm{Hom}(O_{v_\alpha}^\times,\mathbb{F}_p)$. The dimension of the above space is $t_p$ by definition. So we are done.
\end{proof}

\begin{prop}\label{dimension:prop}
The dimension of $\mathbb{T}^1\otimes_{\mathbb{T}^0} H^0(Y(\mathfrak{N}),\mathbb{F}_p)$ is $h_F^+(\mathfrak{N})\cdot t_p$.
\end{prop}

\begin{proof} 
Note that $\mathbb{T}^0 h_{1,\alpha}$ is a vector subspace of $\mathbb{T}^1$ for $\alpha\in\mathrm{Hom}(O_v^\times,\mathbb{F}_p)$ and $v\in T_1$. Also note that $\mathbb{T}^1$ is clearly a finite dimensional vector space.
Thus, by Proposition \ref{Hecke:module:structure} (3), there exists a finite subset $A$ of $\bigsqcup_{v\in T_0}\mathrm{Hom}(O_v^\times,\mathbb{F}_p)$ such that
$$
\mathbb{T}^1=\sum_{\alpha\in A}\mathbb{T}^0h_{1,\alpha}.
$$
If there is an element $\alpha\in A$ such that $h_{1,\alpha}=\sum_{r}c_rh_{1,\alpha_r}$ for some $c_r\in\mathbb{F}_p$ and $\alpha_r\in A$, then $\mathbb{T}^0h_{1,\alpha}=\mathbb{T}^0\sum_r c_r h_{1,\alpha_r}\subset\sum_r \mathbb{T}^0h_{1,\alpha_r}$. Therefore, we may assume that the set $\{h_{1,\alpha}\mid \alpha\in A\}$ is linearly independent.

We would like to show that $\{i_{v,\mathfrak{N}}^*(\alpha)\mid\alpha\in A\}$ is also linearly independent. Let us assume that $\sum_{\alpha\in A}c_\alpha i_{v,\mathfrak{N}}^*(\alpha)=0$ for some constants $c_\alpha$, then $\iota_b^*(h1_a)=0$ for any $a,b\in C(\mathfrak{N})$ by Proposition \ref{computation:hecke:operator}, where $h=\sum_{\alpha\in A}c_\alpha h_{1,\alpha}$. Therefore, $h=0$, which implies that $c_\alpha$'s are all zero. From this, we conclude that $|A|\leq t_p$.

Hence, by Proposition \ref{Hecke:module:structure} (1), $\mathbb{T}^1\otimes_{\mathbb{T}^0} H^0(Y(\mathfrak{N}),\mathbb{F}_p)$ is generated by $h_{z,1}h_{1,\alpha}\otimes 1_1$ for $z\in\mathbb{A}_F^\times/\widehat{O}_F^\times$ and $\alpha\in A$ as a $\mathbb{F}_p$-vector space. Note that $h_{z,1}h_{1,\alpha}\otimes 1_1=h_{1,\alpha}\otimes 1_{z^{-1}(1)}$ by the equation (\ref{hecke:action:representation}). 
%we have the following isomorphism:
%$$
%\mathrm{Hom}(H^0(Y(\mathfrak{N}),\mathbb{F}_p),H^1(Y(\mathfrak{N}),\mathbb{F}_p))\to H^1(E(\mathfrak{N}),\mathbb{F}_p)^{\pi_0Y(\mathfrak{N})^2},\ h\mapsto (\iota_b^*h1_{b^\prime})_{b,b^\prime\in\pi_0Y(\mathfrak{N})}.
%$$
%Then, under this map, a generator $h_z h_{1,\alpha}\otimes 1_a$ of $\mathbb{T}^1\otimes_{\mathbb{T}^0} H^0(Y(\mathfrak{N}),\mathbb{F}_p)$ maps to
%$$
%w_{z,\alpha}:=(\delta_{z(a),b}\alpha\circ i_{v_\alpha})_{b,b^\prime}
%$$
%by Proposition \ref{computation:hecke:operator}, where $v_\alpha$ is an element of $T_0$ such that $\alpha\in\mathrm{Hom}(O_{v_\alpha}^\times,\mathbb{F}_p)$. Note that if $\rho(z_1)\neq\rho(z_2)$, then the vectors $w_{z_1,\alpha}$ and $w_{z_1,\alpha}$ are linearly independent. Also note that if $\alpha\circ i_{v_\alpha}$ and $\alpha^\prime\circ i_{v_\alpha^\prime}$ are linearly independent, then the vectors $w_{z_1,\alpha}$ and $w_{z_1,\alpha^\prime}$ so are. In conclude, 
Therefore, the dimension of $\mathbb{T}^1\otimes_{\mathbb{T}^0} H^0(Y(\mathfrak{N}),\mathbb{F}_p)$ is bounded above by
$$
h_F^+(\mathfrak{N}) \cdot |A|,
$$
and let us recall that $|A|\leq t_p$.
Therefore, by Proposition \ref{dimension:prop:2}, we obtain the desired result.
\end{proof}

From the above propositions, we conclude that the map $\Psi$ is injective. Since the dimension of $H^1(Y(\mathfrak{N}),\mathbb{F}_p)$ is $h_F^+(\mathfrak{N})\cdot r$, we have the implication: if  $t_p=r$, then the map $\Psi$ is an isomorphism.

%Let us refer a special case of {\it Grunwald--Wang theorem}, which is a significant example of the global-local principle in algebraic number theory:

To proceed, we will need the local-global principle for the existence of $p$-th roots. 
\begin{thm}\label{grunwald:wang:thm}
Let $\alpha$ be an element of $F^\times$. Then, if $\alpha\in (F_v^\times)^p$ for almost all primes $v$, then $\alpha\in (F^\times)^p$.
\end{thm}
\begin{proof}
It is a special case of the more general result for $n$-th roots, known as the Grunwald--Wang theorem. For its proof, see {\rm Artin-Tate \cite[Ch.\,X, \textsection\,1, Thm\,1]{artin2008class}}.
%{\rm Artin-Tate \cite[Chapter X, \textsection 1, Theorem 1]{artin2008class}}.
\end{proof}

%\begin{remark}
%For a treatment of the Grundwald--Wang theorem If we replace $p$ in the statement of the above theorem by an integer $n$ divisible by a high-power of $2$, then there is a counter-example on this theorem, which is discovered by {\rm Shinghao Wang}. We omit the exact statement of this theorem for this case as we are dealing with the case of rational prime $p$ in the present paper. For details, see {\rm Artin-Tate \cite[Chapter X, \textsection 1, Theorem 1]{artin2008class}}.
%\end{remark}

\begin{comment}
Clearly, $\delta\leq r$. If $p$ is odd, then $\delta$ is equal to the dimension of 
$$\mathrm{Hom}(E(O_F)/\mu_{F,p}E(\mathfrak{N}),\mathbb{F}_p)$$ since $O_F^\times/E(O_F)$ is a $2$-group and both $O_F^\times$ and $E(O_F)$ contain $\mu_{F,p}$, where $\mu_{F,p}$ is the set of $p$-power roots of unity in $F$. 
\begin{cor}
If a prime $p$ does not divide $\left | O_F^\times/E(\mathfrak{N})\right |$ for a fixed level $\mathfrak{N}$, then $\delta=0$.
\end{cor}
\end{comment}

Denote by $i_v\colon O_F^\times \to \kappa_v^\times$ the reduction map. %map sending $\varepsilon\in O_F^\times$ to $\varepsilon_v\ {\rm mod}\ v\in \kappa_v^\times$.
From Theorem~\ref{grunwald:wang:thm}, we deduce the following lemma which will be crucial in the proofs of main theorems:
\begin{lem}\label{dimension:lemma}
There is a finite subset $S$ of $T_1$ such that the map
$$
\prod_{v\in S} i_{v}^*:\prod_{v\in S}\mathrm{Hom}(\kappa_v^\times,\mathbb{F}_p)\to\mathrm{Hom}(O_F^\times,\mathbb{F}_p)
$$
is an isomorphism.
\end{lem} 
\begin{proof}
Note that each element of $O_F^\times\otimes_\mathbb{Z}\mathbb{F}_p\cong O_F^\times/(O_F^\times)^p$ can be represented by $\varepsilon\otimes 1$ for some $\varepsilon\in O_F^\times$.
Let $\varepsilon\otimes 1\in O_F^\times\otimes_\mathbb{Z}\mathbb{F}_p$. 
Let us consider the map 
$$
\bigoplus_{v\in T_1}i_{v}\otimes 1: O_F^\times\otimes_\mathbb{Z}\mathbb{F}_p\rightarrow\bigoplus_{v\in T_1}\kappa_v^\times\otimes_\mathbb{Z}\mathbb{F}_p
$$
induced from the maps $i_{v}$ for $v\in T_1$.
If $\varepsilon\otimes 1$ is trivial under this map, then $\varepsilon\in (\kappa_v^\times)^p$ for $v\in T_0$ as $(\kappa_v^\times)^p=\kappa_v^\times$ if $v\in T_0-T_1$. Since the $p$-power map is an automorphism on $1+vO_v$ if $p\nmid N(v)$, we observe that $\varepsilon\in (O_v^\times)^p$ for $v\in T_0$. Since the complement of $T_0$ is a finite set, $\varepsilon\in (F^\times)^p\cap O_F^\times=(O_F^\times)^p$ by Theorem \ref{grunwald:wang:thm}, which implies that
$$
1 \xrightarrow{}  O_F^\times\otimes_\mathbb{Z}\mathbb{F}_p \xrightarrow{\oplus_{v}i_v\otimes 1}  \bigoplus_{v\in T_1}\kappa_v^\times\otimes_\mathbb{Z}\mathbb{F}_p 
$$
is exact. Since each $\kappa_v^\times\otimes_\mathbb{Z}\mathbb{F}_p$ is one dimensional, we can find a finite subset $S$ of $T_1$ such that the map $\bigoplus_{v\in S}i_v\otimes 1$ is an isomorphism. Taking the functor $\mathrm{Hom}(-,\mathbb{F}_p)$, the assertion of the lemma follows.
%can prove (1) as the functors $\mathrm{Hom}(-,\mathbb{F}_p)$ and $\mathrm{Hom}_{\mathbb{F}_p}(-\otimes\mathbb{F}_p,\mathbb{F}_p)$ are naturally isomorphic.
\end{proof}

Let $r_p$ denote the $\mathbb F_p$-dimension of $\operatorname{Hom}(O_F^\times, \mathbb F_p)$. Note that $r_p=r$ unless $F$ contains a primitive $p$-th root of unity, in which case $r_p=r+1$.
Denote by $\delta_p$ the $\mathbb F_p$-dimension of $\mathrm{Hom}(O_F^\times/E(\mathfrak{N}),\mathbb{F}_p)$. 

\begin{lem}
Image of the map 
$$
\mathrm{Hom}(O_F^\times,\mathbb{F}_p)\to\mathrm{Hom}(E(\mathfrak{N}),\mathbb{F}_p),
$$
induced by the inclusion $E(\mathfrak{N})\subset O_F^\times$, has dimension
$r_p-\delta_p$.
\end{lem}
\begin{proof}

We have the following exact sequence:
$$
\begin{tikzcd}
1 \arrow{r} & \mathrm{Hom}(O_F^\times/E(\mathfrak{N}),\mathbb{F}_p) \arrow{r} &  \mathrm{Hom}(O_F^\times,\mathbb{F}_p) \arrow{r} & \mathrm{Hom}(E(\mathfrak{N}),\mathbb{F}_p),
\end{tikzcd}
$$
which is induced by the inclusion $E(\mathfrak{N})\subset O_F^\times$.
So we are done.
\end{proof}

\begin{remark}
Since $O_F^\times/E(\mathfrak{N})$ is a finite abelian group, $\delta_p>0$ if and only if $p$ divides $\left |O_F^\times/E(\mathfrak{N})\right |$. In particular, when $r>0$ and $p$ does not divide $\left | O_F^\times/E(\mathfrak{N}) \right|$, we have $r_p-\delta_p >0$. 
\end{remark}

Now we can prove the following theorems:
\begin{thm}\label{spectral:degeneracy}
If $\delta_p<r_p$, then any Hecke eigensystems in $H^0(Y(\mathfrak{N}),\overline{\mathbb{F}}_p)$ occurs again in $H^1(Y(\mathfrak{N}),\overline{\mathbb{F}}_p)$. 
%Furthermore, $p$ does not divide $|O_F^\times/E(\mathfrak{N})|$, then the set of eigensystems of $H^0(Y(\mathfrak{N}),\mathbb{F}_p)$ and one of $H^1(Y(\mathfrak{N}),\mathbb{F}_p)$ coincide.
\end{thm}

\begin{proof}
We consider the derived Hecke action with coefficients in $\overline{\mathbb{F}}_p$.
The result will be equal to the base change of the derived Hecke action with $\mathbb F_p$-coefficients, as explained in \cite[\textsection 2.12]{venkatesh2016derived}.
%, the derived Hecke actions of coefficients $\mathbb{F}_p$ and one of $\overline{\mathbb{F}}_p$ are compatible via the functorial maps of cohomology groups induced by the inclusion $\mathbb{F}_p\to\overline{\mathbb{F}}_p$.
Also note that the natural map induces an isomorphism $H^j(Y(\mathfrak{N}),\overline{\mathbb{F}}_p)\simeq H^j(Y(\mathfrak{N}),\mathbb{F}_p)\otimes_{\mathbb{F}_p}\overline{\mathbb{F}}_p$, because $\mathbb F_p \to \overline{\mathbb{F}}_p$ is flat.

Let $\lambda:\mathbb{T}^0\rightarrow\overline{\mathbb{F}}_p$ be a Hecke eigensystem on $H^0(Y(\mathfrak{N}),\overline{\mathbb{F}}_p)$ and $c\in H^0(Y(\mathfrak{N}),\overline{\mathbb{F}}_p)$ an eigenvector for the given Hecke eigensystem. 
Note that $c$ can be written as a linear combination of $1_a$'s over $\overline{\mathbb{F}}_p$. By Lemma \ref{dimension:lemma}, there exists $v\in T_1$ such that $i_{v,\mathfrak{N}}^*(\alpha)\neq 0$ for some $\alpha\in H^1(O_v^\times,\mathbb{F}_p)$, so that $h_{1,\alpha}c\neq 0$ by Proposition \ref{computation:hecke:operator} and $h_{1,\alpha}c\in H^1(Y(\mathfrak{N}),\overline{\mathbb{F}}_p)$. Then, for any $h_0\in\mathbb{T}^0$,  $h_0h_{1,\alpha}c=h_{1,\alpha}h_0c=\lambda(h_0)h_{1,\alpha}c$ by Proposition \ref{hecke:algebra:commutative}. So we are done.
%On the other hand, let $\lambda:\mathbb{T}^0\rightarrow\overline{\mathbb{F}}_p$ be a Hecke eigensystem on $H^1(Y(\mathfrak{N}),\mathbb{F}_p)$ and $c\in H^1(Y(\mathfrak{N}),\mathbb{F}_p)$ an eigenvector for the given Hecke eigensystem. We can choose a subset $A$ of $\sqcup_{v\in T_1}\mathrm{Hom}_{\mathbb{F}_p}(,)$ and elements $c_\alpha\in H^0(Y(\mathfrak{N}),\mathbb{F}_p)$ for $\alpha\in A$ such that $\Psi^{-1}(c)=\sum_{\alpha\in A}h_{1,\alpha}c_\alpha$ due to Proposiiton \ref{Hecke:module:structure} (3) and the finiteness of the dimension of $\mathbb{T}^1$ and $\mathbb{T}^0h_{1,\alpha}$. Note that $c=$ where $c_\alpha=\sum_{a\in C(\mathfrak{N})}c_{\alpha,a}1_a$ by Proposition \ref{computation:hecke:operator}. Then, for any $h_0\in\mathbb{T}^0$,  $h_0h_{1,\alpha}c=h_{1,\alpha}h_0c=\lambda(h_0)h_{1,\alpha}c$ by Proposition \ref{hecke:algebra:commutative}.
\end{proof}

%degeneracy for the case of characteristic 0 does not imply the degeneracy for mod p due the period of Hodge star.

\begin{thm}\label{modp:nonvanishing}
The dimension $t_p$ is equal to $r_p-\delta_p$. In particular, the map $\Psi$ is an isomorphism if and only if $\delta_p=r_p-r$.
\end{thm}

\begin{proof}
Let us recall that $t_p$ is the dimension of the following subspace:
$$
\sum_{v\in T_1}i_{v,\mathfrak{N}}^*\mathrm{Hom}(\kappa_v^\times,\mathbb{F}_p)\subset\mathrm{Hom}(E(\mathfrak{N}),\mathbb{F}_p).
$$
Note that the map $i_{v}$ factors through $i_{v,\mathfrak{N}}$ via the inclusion $E(\mathfrak{N})\subset O_F^\times$. Now the assertion of the theorem follows from Lemma \ref{dimension:lemma}.
\end{proof}

\begin{remark}
For the condition $\delta_p=r_p-r$ in Theorem\,\ref{modp:nonvanishing} to hold true, we have either $\delta_p=r_p-r=1$ or $\delta_p=r_p-r=0$. The latter happens exactly when $F$ contains no primitive $p$-th roots of unity and $p \nmid \left|O_F^\times / E(\mathfrak N)\right|$. The former happens exactly when $F$ contains a primitive $p$-th root of unity and $\delta_p=1$
\end{remark}
Since localization is exact, Theorem~\ref{modp:nonvanishing} implies:
\begin{cor}\label{cor:4.10}
Let $\mathfrak{m}$ be a maximal ideal of $\mathbb{T}^0$. If $p$ does not divide $|O_F^\times/E(\mathfrak{N})|$, we have the following isomorphism:
$$
\mathbb{T}^1_\mathfrak{m}\otimes_{\mathbb{T}^0_\mathfrak{m}}H^0(Y(\mathfrak{N}),\mathbb{F}_p)_\mathfrak{m}\cong H^1(Y(\mathfrak{N}),\mathbb{F}_p)_\mathfrak{m}.
$$
\end{cor}

We conlude this section by deducing our main results. Assume $p$ does not divide the order of $O_F^\times/E(\mathfrak{N})$.
It implie $\delta_p=0$.

\begin{proof}[proof of Theorem A]
Theorem \ref{modp:nonvanishing}, combined with $\delta_p=0$, implies Theorem A.
\end{proof}

\begin{proof}[proof of Theorem B]
Recall that we assumed that $F$ is neither $\mathbb Q$ nor an imaginary quadratic field. In particular, $r_p>0$.
Now Theorem~\ref{spectral:degeneracy} and Corollary~\ref{cor:4.10} imply Theorem B. 
\end{proof}
%Let us give a family of maximal ideals of $\mathbb{T}^0$.
%Let $\psi:\mathrm{Cl}_F^+(\mathfrak{N})\to\mathbb{C}^\times$ be a ray class character modulo $\mathfrak{N}$. 
%Let us define a map $\mathbb{T}^0\to\mathbb{F}_p(\psi)$ by $h_{z,1}\mapsto \psi(z)$, then this map is a $\mathbb{F}_p$-algebra map.
%Note that $H^0(Y(\mathfrak{N}),\mathbb{F}_p)\cong \mathbb{F}_p^{\mathrm{Cl}_F(\mathfrak{N})}$.
%Let $\mathfrak{m}_\psi$ be the kernel of the map, so that $\mathbb{T}^0/\mathfrak{m}_\psi\cong\mathbb{F}_p(\psi)$ and $\mathfrak{m}_\psi$ is a maximal ideal of $\mathbb{T}^0$.
%\end{remark}

\section*{Acknowledgements}
\thispagestyle{empty}
This work was supported by the Global-LAMP program of the National Research Foundation of Korea (NRF) grant funded by the Ministry of Education (No. RS-2023-00301976). 
The first named author is supported by the NRF.\footnote{National Research Foundation of Korea, No. 2020R1C1C1A01006819}
The second named author is supported by the Basic Science Research program of the National Research Foundation of Korea (NRF) grant funded by the Ministry of Education (No. RS-2023-00245291).

\begin{comment}

\end{comment}
\bibliographystyle{plain}
\bibliography{ref}

\end{document}